\documentclass[12pt,reqno]{amsart}
\usepackage[margin=1.28in]{geometry}
\usepackage[british]{babel}
\usepackage{amsfonts,amsmath,amsthm,amssymb,mathtools}
\usepackage{enumitem,colonequals}
\usepackage{caption}
\usepackage{numdef}
\usepackage{mathrsfs}
\usepackage[all]{xy}
\usepackage{hyperref}
\usepackage{subfig, float, multirow}
\usepackage{array}
\usepackage{fancyhdr}
\usepackage[T1]{fontenc}
\usepackage{lmodern}

\usepackage{graphicx}
\graphicspath{{images/}}
\newcommand{\tikzmark}[2]{\tikz[overlay,remember picture,baseline] \node [anchor=base] (#1) {$#2$};}
\newcommand{\DrawArrow}[3][]{%
  \begin{tikzpicture}[overlay,remember picture]
    \draw[-latex, thick, #1] ($(#2.south)+(0,0.3cm)$) -- ($(#3.north)-(0,0.25cm)$);
  \end{tikzpicture}}

\usepackage{tikz}
\usepackage{tikz-cd}
\usetikzlibrary{matrix,arrows,calc}

\newtheorem{theorem}{Theorem}[section]
\newtheorem{lemma}[theorem]{Lemma}
\newtheorem{proposition}[theorem]{Proposition}
\newtheorem{corollary}[theorem]{Corollary}

\newtheorem{cor}[theorem]{Corollary}

\theoremstyle{definition}
\newtheorem*{ack}{Acknowledgements}
\newtheorem{remark}[theorem]{Remark}
\newtheorem{rem}[theorem]{Remark}
\newtheorem{example}[theorem]{Example}

\newtheorem{dfn}[theorem]{Definition}

\newtheorem{assumption}[theorem]{Assumption}

\numberwithin{equation}{section} \numberwithin{figure}{section}
\let\ker\relax
\let\res\relax
\DeclareMathOperator{\ker}{Ker}
\DeclareMathOperator{\res}{Res}

\DeclareMathOperator{\Aut}{Aut}

\DeclareMathOperator{\disc}{Disc}

\DeclareMathOperator{\ch}{char}

\DeclareMathOperator{\Jac}{Jac}
\DeclareMathOperator{\jac}{Jac}

\newcommand{\GL}{\textrm{GL}}

\newcommand{\SL}{\textrm{SL}}
\newcommand\inj{\hookrightarrow}

\newcommand{\lra}{\longrightarrow}

\newcommand{\LL}{\mathscr{L}}
\newcommand{\MM}{\mathscr{M}}
\newcommand{\KK}{\mathcal{K}}

\newcommand\FF{\mathbf{F}}
\newcommand\PP{\mathbf{P}}
\newcommand\ZZ{\mathbf{Z}}
\newcommand\NN{\mathbf{N}}

\newcommand\CC{\mathbf{C}}

\newcommand\Aff{\mathbf{A}}

\renewcommand{\leq}{\leqslant}

\renewcommand{\geq}{\geqslant}

\newcommand{\x}{\mathbf{x}}

\renewcommand{\d}{\mathrm{d}}

\newcommand{\kbar}{{\newoverline[.75]{K}}}
\makeatletter
\newsavebox\myboxA
\newsavebox\myboxB
\newlength\mylenA
\newcommand*\newoverline[2][0.75]{
    \sbox{\myboxA}{$\m@th#2$}%
    \setbox\myboxB\null
    \ht\myboxB=\ht\myboxA
    \dp\myboxB=\dp\myboxA
    \wd\myboxB=#1\wd\myboxA
    \sbox\myboxB{\hspace{2.3pt}$\m@th\overline{\copy\myboxB}$}%
    \setlength\mylenA{\the\wd\myboxA}
    \addtolength\mylenA{-\the\wd\myboxB}
    \ifdim\wd\myboxB<\wd\myboxA
       \rlap{\hskip 0.5\mylenA\usebox\myboxB}{\usebox\myboxA}
    \else
        \hskip -0.5\mylenA\rlap{\usebox\myboxA}{\hskip 0.5\mylenA\usebox\myboxB}
    \fi}
    
\newcommand{\isog}{\varphi}
\newcommand\degree{n}
\newcommand{\isom}{\alpha}
\newcommand{\w}{{\omega}}

\num\newcommand{\map1}{\phi_1}
\num\newcommand{\mapd1}{f_1}
\num\newcommand{\E1}{E_1}
\num\newcommand{\zero1}{O_1}
\num\newcommand{\mapd2}{f_2}
\num\newcommand{\map2}{\phi_2}
\num\newcommand{\E2}{E_2}
\num\newcommand{\zero2}{O_2}  
\num\newcommand{\proj1}{\pi_1}
\num\newcommand{\proj2}{\pi_2}

\num\newcommand{\projpush1}{\pi_{1*}}
\num\newcommand{\projpush2}{\pi_{2*}}
\num\newcommand{\mapdpush1}{f_{1*}}
\num\newcommand{\mapdpush2}{f_{2*}}
\num\newcommand{\mappush1}{\phi_{1*}}
\num

\newcommand{\smallmat}[4]
  {\bigl[ \begin{smallmatrix}
	#1 & #2\\
	#3 & #4
  \end{smallmatrix} \bigr]}

\newcommand{\hesse}{\mathcal{H}}
\newcommand{\Ei}{E_i}
\newcommand{\mapdi}{\mapdd_i}
\newcommand{\mapi}{\mapu_i}

\newcommand{\proji}{\pi_i}
\newcommand{\W}{Weier\-stra{\ss}{ }}
\newcommand\mapu{\phi}
\newcommand\mapdd{f}
\newcommand\invol{\iota}

\newcommand{\bigindex}[1]{_{\hbox{$#1$}}}

\newcommand{\projj[1]}{ \pi_{\hspace{-0.07em} \raisebox{-0.04em}{\resizebox{.33em}{!}{$#1$}}}}

\newcommand{\comp}{\raisebox{.36pt}{\hspace{.2em}\scalebox{.8}{$\circ$}\hspace{.2em}}}
\newcommand{\minus}{-}
\def\chk#1{#1^{\smash{\raisebox{-1pt}{\scalebox{1}[2]{\rotatebox{90}{\textup{\guilsinglleft}}}}}}}
\newdimen\picsize
\title{Families of (3,3)-split Jacobians}
\author[1]{Martin Djukanovi\'c}
\email{djukanovic@gmail.com}
\begin{document}
\begin{abstract}  
We compute all the ``special'' cases of $(3,3)$-split Jacobians and we parametrize the Igusa--Clebsch invariants of curves of genus two whose Jacobian is $(3,3)$-isogenous to a product of two elliptic curves from the Hesse pencil.
\end{abstract}

\maketitle
\section{Introduction} 
If $C$ is a curve of genus two, equipped with an \emph{optimal} covering $\mapu\colon C\to E$ of degree~$n$ of an elliptic curve then there exists an elliptic curve~$E'$ such that the Jacobian~$\jac(C)$ is isogenous to $E\times E'$ via an isogeny of degree $n^2$. Such a Jacobian is said to be $(n,n)$-split. The classical treatment of cases $n\leq 4$, in terms of elliptic integrals, can be found in the works of Legendre, Jacobi, Goursat, and others; e.g. see \cite{jacobi}, \cite{goursat}, \cite{brioschi}, \cite{bolza}. A modern treatment can be found in e.g.~\cite{kuhn}, 
\cite{frey}, 
\cite{freykani}, 
\cite{shaska2}, 
\cite{shaska}, 
\cite{howe}, 
\cite{bruin}, 
\cite{shaska3}, 
\cite{kumar}, dealing with various cases with $n\leq 11$. 
The problem of finding the curve $E'$, given the map $\mapu$, was considered in~\cite{kuhn}. Explicit examples and infinite families appear in~\cite{howe}, \cite{kuhn}, \cite{shaska} for $n=3$, which is the case that is the topic of this paper. In the first three sections, we review known results and offer minor corrections and clarifications. In the remaining sections, we consider a similar problem -- given two elliptic curves $E$ and~$E'$, find the isomorphism class of a genus-$2$ curve $C$ such that~$\jac(C)\sim E\times E'$, if it exists. When dealing with both of these problems, one is not only interested in concrete examples, but also in parametrizations of infinite families. In Section~5 we give a parametrization of the Igusa--Clebsch invariants of~$C$ in terms of two parameters that define a pair of elliptic curves from the Hesse pencil. We believe that this approach can be generalized for various $n\geq 4$.

\subsection*{Notations and definitions}
Throughout the paper,~$K$ denotes a field of characteristic $\ch(K)\neq 2$ and~$\kbar$ denotes an algebraic closure of $K$. Unless otherwise specified, all varieties are projective and defined over~$K$. If~$C$ is a genus-$2$ curve over~$K$, we say that its Jacobian $\jac(C)$ is \emph{split} if it is isogenous over~$K$ to a product of two elliptic curves. By a \emph{covering} we mean a finite, surjective, separable morphism. Given a divisor~$D$ on a $K$-variety, we denote by~$L(D)$ the $K$-vector space of global sections of the invertible sheaf~$\LL(D)$ associated to~$D$. Given a commutative ring~$R$ and polynomials $F,G\in R[x]$, the resultant of~$F$ and~$G$ is denoted by~$\res_x(F,G)$  and the discriminant of $F$ is denoted by~$\disc_x(F)$.

\section{A curve of genus two covering a curve of genus one}
\label{section:coverings}
Let $C$ be a curve of genus two, defined over~$K$ and equipped with a covering $\mapu\colon C\to E$ of degree~$\degree$, where~$E$ is a curve of genus one and~$\degree$ is coprime to~$\ch(K)$. Recall that~$C$ is hyperelliptic since the linear system defined by its canonical divisor~$K_C$ defines a $2$-to-$1$ map to~$\PP^1$, by Riemann--Roch. Let~$\invol$ denote the hyperelliptic involution on~$C$, so that $C/\invol\cong\PP^1$. For a \W point \mbox{$W\in C(\kbar)$}, one can define the Abel--Jacobi map $\varepsilon:C\inj\jac(C)$, given by~\mbox{$P\mapsto[P-W]$,} and embed $C$ into its Jacobian. One also has a corresponding isomorphism $E\cong\jac(E)$, given by~\mbox{$P\mapsto[P-\mapu(W)]$.} The morphism \mbox{$\mapu : C\to E$} induces a group morphism $\mapu_* : \jac(C) \to E$ so that~$\mapu_* \comp \varepsilon = \mapu$. The group morphism~$\mapu_*$ commutes with $[\minus 1]$ and therefore induces a map $\jac(C)/[\minus 1] \to E/[\minus 1]$. Thus we obtain the following commutative diagram
\begin{equation}
\label{diag:mapd2}
\begin{tikzcd}
C \arrow[r,"\mapu"] \arrow{d}{\projj[C]} & E \arrow{d}{\projj[E]}\\
C/\bigindex{\invol} \arrow[r,"\mapdd"] & E/[\minus 1]
\end{tikzcd}
\end{equation}
Here $\pi_C$ and $\pi_E$ denote the canonical maps.

\begin{remark}
The Abel--Jacobi map $\varepsilon$ is in general only defined over a quadratic extension of $K$, but this issue is easily resolved when $\degree$ is odd, as we shall explain below.
\end{remark}

\begin{assumption}
From now on, we will assume that the covering map $\mapu\colon C\to E$ is always \emph{optimal}, which is to say that it does not factor through a non-trivial isogeny.
\end{assumption}

\subsection{Ramification analysis}
Kuhn~\cite{kuhn} analysed the ramification of the map~$\mapdd$. We recall the main results. The map $\pi_C$ has six geometric ramification points, whereas~$\pi_E$ has four, by Riemann--Hurwitz. These are, of course, the points fixed by the corresponding involutions. Let $W_1,\dots ,W_6$ denote the ramification points of~$\pi_C$ and let $T_1,\dots,T_4$ denote the ramification points of~$\pi_E$. Moreover, let~$w_1,\dots, w_6$ and~\mbox{$t_1,\dots,t_4$} denote their respective images under the corresponding canonical maps~$\pi_C$ and~$\pi_E$. It is clear from the above that $\{\mapu(W_i)\}\subseteq\{T_j\}$ and~$\{\mapdd(w_i)\}\subseteq\{t_j\}$. 
\par

By Riemann--Hurwitz, we have that $\mapu$ has ramification degree $2$, meaning that $\mapu$ either doubly ramifies at two distinct points or it has one triple ramification point. We distinguish two cases -- either the ramification of~$\mapu$ occurs above some~$T_j$ or it does not. These are referred to as the ``special'' case and the ``general'' case, respectively. As~$\invol$ acts on the ramification divisor $R_\mapu$ of $\mapu$, if there are two distinct ramification points, they cannot lie above two distinct~$T_j$. In the generic case, the map \mbox{$\pi_E\comp\mapu=\mapdd\comp\pi_C$} ramifies at~$4n$ double points that lie above the~$T_j$. Since~$\pi_C$ ramifies at six double points, we have that~$\mapdd$ ramifies at $\tfrac{1}{2}(4n-6)=2n-3$ double points above the $t_j$, none of which is any of the $w_i$. By Riemann--Hurwitz, $\mapdd$ has ramification degree~\mbox{$2n-2$}, which means that there is one more doubly ramified point that does not lie above the~$t_j$. In the special case, all of the ramification lies above the~$t_j$. Since~$\mapdd$ is a finite map between smooth varieties, it is flat and every fibre has $\degree=\deg\mapdd$ geometric points, counting with multiplicities.
\begin{lemma}
\label{lemma:kuhn}
With the notations as above, for every $i\in\{1,2,\dots,6\}$ the divisor $\mapdd^{*}(\sum_{j=1}^{4}t_j)$ contains $w_i$ with odd multiplicity and any other points with even multiplicity.
\end{lemma}
\begin{proof}
See Lemma in \S1 of \cite{kuhn} or Lemma 2.1 in \cite{freykani}.
\end{proof}

\begin{lemma}
\label{lemma:main_lemma}
Let $C$ be a curve of genus two and let $\map1\colon C\to \E1$ be an optimal covering of an elliptic curve $(\E1,\zero1)$ with $\deg\map1=n$ coprime to $\ch(K)$. Then, possibly after extending the base field, there exists an elliptic curve $(\E2,\zero2)$, an optimal covering $\map2\colon C\to \E2$, and an isogeny $\isog\colon \E1\times \E2\to\jac(C)$ such that:
\begin{enumerate}[label=(\arabic*)]
\item $\deg\map2=n$;
\item $\isog=\map1^*+\map2^*$;
\item $\deg\isog=n^2$;\item $\ker(\isog)\cong \E1[n]\cong \E2[n]$.
\end{enumerate}
\end{lemma}
\begin{proof}
The covering map $\map1$ induces an embedding $\map1^*\colon\E1\inj\jac(C)$, with respect to an isomorphism $\E1\cong\jac(\E1)$. The elliptic curve $\E2$ is given as $\ker(\mappush1)\subset\jac(C)$, which is connected because $\map1$ is optimal. Let $\varepsilon\colon C\inj\jac(C)$ be an embedding, not necessarily defined over $K$. Recalling that Jacobians are (canonically) self-dual, let $\eta\colon\jac(C)\to\E2$ denote the map dual to the inclusion $\E2\inj\jac(C)$. The covering map $\map2\colon C\to\E2$ is then obtained as the composition $\eta\comp\varepsilon$. The isogeny \mbox{$\isog\colon\E1\times\E2\to\jac(C)$} is given by $\isog=\map1^*+\map2^*$ and its kernel is the image of~$\Ei[n]$ under the embedding $\mapi^*\colon\Ei\inj\jac(C)$ for both $i\in\{1,2\}$. For more details, see Lemma in~\S2 of \cite{kuhn} or Lemma~1.6 in~\cite{thesis}.
\end{proof}
\begin{dfn}
With the assumptions of Lemma~\ref{lemma:main_lemma}, the Jacobian $\jac(C)$ is said to be \emph{$(\degree,\degree)$-split} and the curves~$\E1$ and~$\E2$ are said to be \emph{glued along their $\degree$-torsion}.
\end{dfn}
\begin{rem}
The constructions in Lemma~\ref{lemma:main_lemma} depend on the choice of the embedding $\varepsilon\colon C\inj\jac(C)$, which need not be $K$-rational.
\end{rem}

\subsection{Optimal coverings of odd degree}
Let $\map1\colon C\to\E1$ be an optimal covering of odd degree $\degree$. Let~$\proji\colon\Ei\to\PP^1$ be the canonical maps, let~$T_j$ be the geometric ramification points of~$\proj1$ and let $t_j=\proj1(T_j)$. It follows from Lemma~\ref{lemma:kuhn} that there is a unique ramification point of $\proj1$, say~$T_4$, such that exactly three of the~$W_i$ map to it under $\map1$. Moreover, there is exactly one~$W_i$ above each point in $\{T_1,T_2,T_3\}$. We index the points so that $W_1,W_2,W_3$ lie above~$T_4$. It follows that the divisors $W_1+W_2+W_3$ and $W_4+W_5+W_6$ are~$K$-rational and that the point $T_4$ is $K$-rational. Analogous statements hold for the points~$w_i$ and~$t_j$. Thus we conclude that the curve~$C$ admits an affine plane model~\mbox{$y^2=P(x)Q(x)$,} where~\mbox{$P(x),Q(x)\in K[x]$} are cubics whose roots are~\mbox{$\{w_1,w_2,w_3\}$} and~\mbox{$\{w_4,w_5,w_6\}$}, respectively. Since the class of the canonical divisor~\mbox{$K_C\sim 2W_i$} is~\mbox{$K$-rational}, so is the class of the divisor~\mbox{$W_1-W_2+W_3$}. Since the latter is a divisor of degree one and invariant under the action of the hyperelliptic involution, it follows that~$\map1$ induces a canonical \mbox{$K$-rational} embedding~\mbox{$C\inj\jac(C)$}, given by
\begin{equation}
\label{eq:embedding}
 P\mapsto[P-W_1+W_2-W_3],
\end{equation}
 which is compatible with the isomorphism~$\E1\cong\jac(\E1)$, given by~\mbox{$P\mapsto[P-T_4]$}, and the involutions. We summarize with the following (see \S4 in \cite{kuhn} for details).

\begin{lemma}
There is a canonical and $K$-rational choice for $E_2$ and the associated morphisms.
\end{lemma}

\begin{dfn}
\label{def:complement}
We will implicitly assume the embedding~\eqref{eq:embedding} and say that $\E2$, $\map2$, and~$\mapd2$ are \emph{complementary} to $\E1$, $\map1$, and $\mapd1$, respectively.
\end{dfn}

\begin{lemma}
\label{lemma:kuhn_odd}
The roles of the divisors $w_1+w_2+w_3$ and $w_4+w_5+w_6$ are exchanged between the complementary maps~$\mapd1$ and~$\mapd2$, and we have $\mapdpush1(w_1+w_2+w_3)=3\projpush1(\zero1)$ and $\mapdpush2(w_4+w_5+w_6)=3\projpush2(\zero2)$, and hence also ${\mapd1}_*(w_4+w_5+w_6)=\projpush1(\E1[2]\setminus\{\zero1\})$ and $\mapdpush2(w_1+w_2+w_3)=\projpush2(\E2[2]\setminus\{\zero2\})$, where the identity element~$\zero1\in\E1(K)$ is the point~$T_4$.
\end{lemma}

\section{Covering maps of degree three}
We will now focus on the case~$\degree=3$. We therefore assume that $\ch(K)\not\in\{2,3\}$. It follows from Lemma~\ref{lemma:kuhn} and its preceding paragraph that the ramification of the map $\mapd1\colon C\to\E1$ is restricted to one of two possibilities. This ultimately leads to a parametrization of the $\kbar$-isomorphism invariants of the curves involved. We will deal with the generic case first.

\subsection{The generic case}
\label{sub:generic}
The results in this subsection appear in~\cite{kuhn}, but here we go through the derivation in full detail. Let \mbox{$\map1\colon C\to\E1$} be an optimal covering of degree~$3$. In the generic case, the map $\mapd1$ is doubly ramified at a $K$-rational point that is not in any of the fibres $\mapd1^*(t_j)$. Let us denote the image of this point under $\mapd1$ by~$t_0$. Since~$t_0$ and~$t_4$ are $K$-rational and we are interested in isomorphism classes, we may and do assume that $t_0=0$, $t_4=\infty$, and $\mapd1^{*}(0)=2\cdot0+\infty$. In summary, we assume that the ramification of~$\mapd1$ is as depicted in Fig.~\ref{fig:fig1}, where the unramified points above $t_1,\dots,t_4$ are the~$w_i$ (depicted are the ramification indices of the points in the fibres of $\mapd1$ above the $t_j$).
\begin{figure}[H]%
\hspace{-2.35em}
\centering
\parbox{1em}{
\raisebox{8.4em}{
\begin{tabular}{c}
$\infty$ \\
\raisebox{-0.7em}{$0$}
\end{tabular}
}
}
\;\;
\parbox{\picsize}{%
\centering
\includegraphics[width=\picsize]{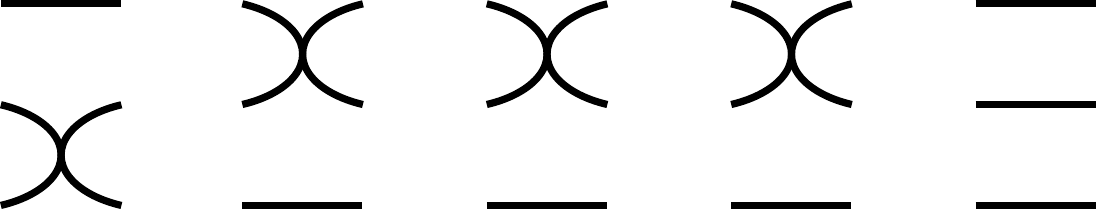}\\

\tikzmark{tikzmark1}{}\\
\hspace{0.6cm}$\mapd1$\\
\tikzmark{tikzmark2}{}\\
\includegraphics[width=\picsize]{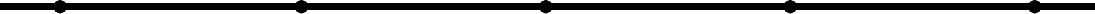}\\
\hspace{.05em} $0$ \hspace{1.41em} $t_1$ \hspace{1.5em}  $t_2$ \hspace{1.5em}  $t_3$ \hspace{1.5em}  $\infty$
}\DrawArrow[->]{tikzmark1}{tikzmark2}
\caption{Ramification of the map $\mapd1$ in the generic case.}
\label{fig:fig1}
\end{figure}
\par
\noindent
In other words, we assume, without loss of generality, that
$$
\mapd1(x)=\dfrac{x^2}{P(x)},
$$
where $P(x)=x^3+ax^2+bx+c\in K[x]$ has roots $w_1,w_2,w_3\in\kbar$. The~$w_i$ are pairwise distinct and none of them equals zero. This can be expressed as
\begin{align*}
\res_x(x,P(x))&=c\neq0,\\
\disc_x(P(x))&=a^2 b^2 - 4 b^3 - 4 a^3 c + 18 a b c - 27 c^2\neq0.
\end{align*}
The pullback of $t_1+t_2+t_3$ corresponds to the roots of $D(x)^2Q(x)$, where~$D(x)$ and~$Q(x)$ are cubics in $K[x]$. Moreover, the roots of~$D(x)$ are the ramification points distinct from~$0$, and the roots of~$Q(x)$ are $w_4,w_5,w_6$. Since
$$
\frac{\d\mapd1}{\d x}(x)=-\frac{x(x^3- b x -2 c )}{P(x)^2}
$$
and the roots of the numerator correspond precisely to the doubly ramified points of~$\mapd1$, we can take $D(x)=x^3- b x -2 c$. The ramification points are again pairwise distinct so we have
\begin{equation}
\disc_x(D(x))=4 (b^3 - 27 c^2)\neq0.
\end{equation}
From this we can calculate the nonic $D(x)^2Q(x)$ whose roots correspond to the divisor $\mapd1^*\comp{\mapd1}_*(d_1+d_2+d_3)$, where the $d_i$ are the roots of $D(x)$. In particular, we have the following equality, up to multiplication by constants,
$$
\res_y(x^2P(y)-y^2P(x),D(y))=D(x)^2Q(x).
$$
This resultant is easily found to equal
$$c (x^3- b x -2 c )^2 (4 c x^3+ b^2 x^2 +  2 b c x + c^2),
$$
whence we can take $Q(x)=4 c x^3+ b^2 x^2 +  2 b c x + c^2$. It follows that, up to quadratic twists, the curve~$C$ admits an affine plane model given by
\begin{equation}
\label{eq:generic_C}
y^2=P(x)Q(x)=(x^3+ax^2+bx+c)(4 c x^3+ b^2 x^2 + 2 b c x + c^2).
\end{equation}
By Lemma~\ref{lemma:kuhn_odd}, we may assume
$$
\mapd2(x)=\frac{(x+d)^2(x+e)}{4 c x^3+ b^2 x^2 + 2 b c x + c^2}
$$
for some $d,e\in K$, up to multiplication by a constant. To determine $d$ and $e$, we apply the procedure used to obtain $Q(x)$ from~$\mapd1$ to the map~$\mapd2$. In doing so, we must ultimately obtain a cubic polynomial~$R(x)$ that is a multiple of~$P(x)$, by Lemma~\ref{lemma:kuhn_odd}. Working over the field $K(a,b,c,d,e)$, we can compute the polynomial~$R(x)$ and perform Euclidean division on~$P(x)$ and~$R(x)$. By the argument above, the remainder must equal zero and we obtain three polynomial equations over the ring~$K[a,b,c,d,e]$. More details can be found in~\cite{thesis}, where the computations are performed over $K(a,b,c)[d,e]$. The argument here is only slightly different. Let~$I\subset K[z,a,b,c,d,e]$ denote the ideal generated by
$$
1-z\cdot(d-e)\cdot P(0)\cdot Q(-d)\cdot Q(-e)\cdot\disc_x(P(x))\cdot\disc_x(Q(x))
$$
and the coefficients of the remainder obtained by dividing~$R(x)$ and~$P(x)$. Eliminating the variable $z$ and computing the primary decomposition of the corresponding elimination ideal gives, among others, the following two equations:
\begin{align*}
& bd - 3c = 0,\\
& acd - 4ace + b^2e - bc + 3cde = 0.
\end{align*}
We therefore take
\begin{equation}
\mapd2(x)=\frac{(bx+3c)^2((b^3 - 4abc + 9c^2)x + b^2c - 3ac^2)}{4 c x^3+ b^2 x^2 + 2 b c x + c^2},
\label{eq:d_and_e}
\end{equation}
which also covers the cases when $\infty$ is a zero of $f_2$. Now one can determine the modular invariants of~$\E1$ and~$\E2$. An affine plane model for $\E1$ can be determined, up to quadratic twists, by requiring that the set of branch points of the canonical map $\proj1$ is $\{t_1,t_2,t_3,\infty\}$, i.e.~$\infty$ and the image under~$\mapd1$ of the three roots of~$Q(x)$. Likewise, an affine plane model for~$\E2$ can be determined by requiring that~$\proj2$ ramifies above~$\infty$ and the image under $\mapd2$ of the three roots of~$P(x)$. The corresponding cubics can be obtained from the resultants \mbox{$\res_y(xP(y)-y^2,Q(y))$} and \mbox{$\res_y(xQ(y)-(y+d)^2(y+e),P(y))$} and are omitted here. The $j$-invariants of the two elliptic curves can then be obtained from the cubics by a direct computation. We obtain the following two expressions, which appear in~\S6 of~\cite{kuhn}:
\begin{equation*}
j(\E1)=\frac{16 (a^2 b^4 + 12 b^5 - 126 a b^3 c + 216 a^2 b c^2 + 405 b^2 c^2 - 972 a c^3)^3}{(b^3 - 27 c^2)^3 (a^2 b^2 - 4 b^3 - 4 a^3 c + 18 a b c - 27 c^2)^2},
\label{jinv33}
\end{equation*}
\begin{equation*}
j(\E2)=\frac{256 (a^2 - 3 b)^3}{a^2 b^2 - 4 b^3 - 4 a^3 c + 18 a b c - 27 c^2}.
\end{equation*}

\subsection{The special cases}
In this subsection we will deal with the cases in which one or both of the maps $\mapdi$ are special, i.e. there is a triple ramification point above the branch locus of $\proji$.

\subsubsection{First map special}
\label{subsub:first_special}
Suppose that $\mapd1$ is special and $\mapd2$ is not. By passing to a quadratic extension of~$K$ if necessary and applying a suitable automorphism of~$\PP^1$, we may and do assume that $\mapd1$ has a triple zero at 0 and a simple pole at~$\infty$, so that
$$
\mapd1(x)=\frac{x^3}{x^2+ax+b},\quad\mapd2(x)=\frac{(x+c)^2(x+d)}{Q(x)},
$$
The argument used in the previous subsection gives $Q(x)=(a^2 - 4 b) x^3 - 2 a b x^2 - 3 b^2 x$ and therefore, up to quadratic twists,~$C$ admits an affine plane model given by
$$
y^2=x((a^2 - 4 b) x^2 - 2 a b x - 3 b^2)(x^2+ax+b).
$$
As before, we impose the generic ramification on~$\mapd2$ and use Lemma~\ref{lemma:kuhn_odd}. The difference in this case is that~$R(x)$ is a priori a cubic so we must identify its leading coefficient with zero and then find the remainder by dividing by~$P(x)$. This ultimately gives
\begin{align*}
& a c - 3 b  = 0,\\
& 3 a d - 3 b + c^2 - 4 c d = 0.
\end{align*}

It follows that we can take
\begin{equation*}
\label{eqn:f2_first_special}
\mapd2(x)=\frac{
(a x + 3 b)^2 (a (a^2 - 4 b) x + b (a^2 - 3 b))
}{
x ((a^2 - 4 b) x^2 - 2 a b x - 3 b^2)
}.
\end{equation*}
Simple resultant computations finally give models of~$\E1$ and~$\E2$ from which we obtain
$$
j(\E1)=\frac{16 (16 a^6 - 144 a^4 b + 405 a^2 b^2 - 324 b^3)^3}{729 b^4 (a^2 - 3 b)^3 (a^2 - 4 b)^2},
\quad
j(\E2)=\frac{256 (a^2 - 3 b)^3}{b^2 (a^2 - 4 b)}.
$$
A simpler parametrization can be obtained by reversing the roles of the maps, as below.

\subsubsection{Second map special}
\label{subsub:second_special}
Suppose that $\mapd2$ is special and $\mapd1$ is not. This is analogous to the general case and we may and do assume that
$$
\mapd1(x)=\frac{x^2}{x^3 + ax^2 + bx + c},\quad\mapd2(x)=\frac{(x+d)^3}{4cx^3 + b^2x^2 + 2bcx + c^2}.
$$
Following the same argument, we obtain:
\begin{align*}
& b d - 3 c = 0,\\
& -2 b + 3 a d - 3 d^2 = 0,
\end{align*}
whence we can take
$$
\mapd1(x) = \frac{x^2}
{(b x + 3 c)(9 c x^2+ 2 b^2 x + 3 b c)},
\quad
\mapd2(x) = \frac{(b x+3c)^3}
{4 c x^3 + b^2 x^2 + 2 b c x + c^2}.
$$
Therefore a quadratic twist of $C$ admits an affine plane model given by
$$
y^2 = (b x + 3 c)(9 c x^2+ 2 b^2 x + 3 b c)(4 c x^3 + b^2 x^2 + 2 b c x + c^2).
$$
The $j$-invariants of $\E1$ and $\E2$ are then easily found to be
\begin{equation}
\label{eqn:second_special_j}
j(\E1) = \frac{64 b^3}{c^2},
\quad
j(\E2) = \frac{64(4b^3 - 27c^2)^3}{729 b^3 c^4}.
\end{equation}
Note that in all cases one must assume that the numerator and the denominator of~$\mapdi$ have no common factors with one another. These conditions can be expressed as the non-vanishing of the corresponding resultants, omitted here.

\subsubsection{Both maps special}
\label{subsub:both_special}
Finally, suppose that both $\mapd1$ and $\mapd2$ are special. Without loss of generality, we assume
$$
\mapd1(x)=\frac{x^3}{x^2+ax+b},
\quad
\mapd2(x)=\frac{1}{Q(x)},
$$
where $b\neq 0$ and $a^2-4b\neq 0$. Applying the usual argument to~$\mapd1$, we get
\begin{equation*}
Q(x)=(a^2-4b)x^3-2abx^2-3b^2x.
\end{equation*}
Applying the same argument to $\mapd2$, we conclude that $x^2+ax+b$ must divide
$$
3(a^2 - 4 b)^2 x^2 - 4ab(a^2 - 4 b) x-16  b^2(a^2 - 3 b)
$$
Dividing the latter by the former gives the remainder
$$
-a(3a^2-8b)(a^2-4b)x-a^2b(3a^2-8b).
$$
Given that $a^2-4b\neq 0$ and $b\neq 0$, the remainder is identically zero if and only if~$a=0$ or $b=3a^2/8$. For $a=0$, we obtain
$$
\mapd1(x)=\frac{x^3}{x^2+b},\quad \mapd2(x)=\frac{1}{4x^3+3bx},\quad j(\E1)=j(\E2)=1728.
$$
For $b=3a^2/8$, we obtain
$$
\mapd1(x)=\frac{x^3}{8x^2+8ax+3a^2},
\quad
\mapd2(x)=\frac{1}{32x^3+48ax^2+27a^2x},
$$
$$
j(\E1)=j(\E2)=-\frac{873722816}{59049}=-\frac{2^6\cdot 239^3}{3^{10}}.
$$
\par
This shows that there are exactly two pairs of isomorphism classes of $\E1$ and $\E2$ such that two covering maps $C\to\Ei$ of degree three have a triple ramification point and such that the fibre with three \W points of one covering has no points in common with the fibre with three \W points of the other covering. However, only the case $j(E_1)=j(E_2)=1728$ corresponds to a pair of complementary curves and coverings in the sense of Definition~\ref{def:complement}. In the other case we have $E_1=E_2$ and one special covering is obtained from the other one by composing with an automorphism of $C$. A complementary pair of curves with~$j(E_1)=-873722816/59049$ and a special covering $C\to E_1$ is obtained by putting $a=3a^2/8$ in~\S\ref{subsub:first_special}, which yields a generic covering $C\to E_2$ with~\mbox{$j(E_2)=64/9$.} This is elaborated upon in the Appendix (Example~\ref{example:bad_conclusion}).\par
Therefore there is a unique pair of isomorphism classes with two complementary special coverings of degree three. There have been conflicting claims in the literature about the number of such pairs (cf.~\cite{kuhn}, \cite{shaska}) and we hope that this detailed exposition settles the matter -- the statement in~\S6 of~\cite{kuhn} is correct.
\begin{remark}
Suppose that a curve of genus two covers an elliptic curve of \mbox{$j$-invariant~$J_1$} ``generically'' $3$-to-$1$ and that it covers the complementary curve, of \mbox{$j$-invariant~$J_2$,} ``specially'' $3$-to-$1$. Then it is readily verified, using the parametrizations above, that $F(J_1,J_2)=0$, where
\begin{equation}
\label{eq:param_curve}
F(X,Y):= X^3 - 1296 X^2 - 729 X Y + 559872 X - 80621568.
\end{equation}
The \mbox{$j$-invariant} pairs that are given in~\cite{shaska} are then obtained as the intersection of the curves~ \mbox{$F(X,Y)=0$} and \mbox{$F(Y,X)=0$.} However, the points of the intersection do not correspond to the \mbox{$j$-invariant} pairs of curves whose $3$-to-$1$ coverings by a genus-$2$ curve are both special since~\eqref{eq:param_curve} is obtained under the assumption that one of the coverings is generic. For example, the point $(1728,1728)$ lies on both curves. To see that it is on the first curve, we can take $(b,c)=(3t^2,- t^3)$ in~\eqref{eqn:second_special_j}, for $t\neq 0$. However, these parameters do not define a hyperelliptic curve because $t$ is a multiple root of both~$P(x)$ and~$Q(x)$.
\end{remark}

\begin{remark}
Our choices of ramification points and parametrizations are the same as those made by Kuhn~\cite{kuhn}; this provides some context for the analysis of the special cases. However, in terms of elegance and simplicity, the parametrization given in~\cite{howe} is superior. This is particularly evident when it comes to explicitly writing down the elliptic curves $E_i$ and the degree-$3$ morphisms $C\to E_i$ (see~\S\ref{app:explicit_covering} in the Appendix).
\end{remark}

\section{Gluing two elliptic curves}
In the previous section we started with an optimal covering $C\to \E1$ of degree~$3$ and constructed the complementary curve~$\E2$. In the remaining sections we adopt a different approach. We start with two elliptic curves~$\E1,\E2$ and construct a curve of genus two whose Jacobian is isogenous to \mbox{$\E1\times\E2$} via an isogeny whose kernel is prescribed. This approach can be found in~\cite{freykani}. First we recall some definitions and classical results.
\par
Let $A$ be an abelian variety over $K$ and let $\lambda\colon A\to\chk{A}$ be a polarization. Suppose that $m\in\ZZ$ is coprime to $\ch(K)$ and such that $\ker(\lambda)\subset A[m]$. Let $$e_m\colon A[m](\kbar)\times\chk{A}[m](\kbar)\to\mu_m$$ denote the Weil pairing. Then we can associate to $\lambda$ a skew-symmetric pairing 
$$
e_\lambda\colon\ker(\lambda)\times\ker(\lambda)\to\mu_m
$$
that is defined for any pair $(P,Q)$ of geometric points as $e_\lambda(P,Q) = e_m( P, \lambda(R))$, where $R$ is such that $[m]R=Q$. This does not depend on $R$ or $m$ (see~\S16 in \cite{milneav}).
\par
\begin{lemma}
\label{lemma:mumford}
Let $\varphi\colon A\to B$ be an isogeny whose degree is coprime to $\ch(K)$ and let $\lambda\colon A\to\chk{A}$ be a polarization induced by a line bundle~$\LL$. Then the following are equivalent:
\begin{enumerate}
\item There exists a line bundle $\MM$ on $B$ such that $\LL=\varphi^*(\MM)$, inducing a polarization $\lambda'\colon B\to\chk{B}$,
\item $\ker(\varphi)\subset\ker(\lambda)$ and $e_\lambda$ is trivial on~$\ker(\varphi)\times\ker(\varphi)$.
\end{enumerate}
\end{lemma}
\begin{proof}
See Proposition 16.8 in \cite{milneav} or Theorem 2 and its Corollary in~\S23 of \cite{mumford}.
\end{proof}
\begin{cor}
\label{cor:mumford}
Let $\map1\colon C\to E_1$ be an optimal covering of an elliptic curve by a curve of genus two, such that $\deg\map1=\degree$ is coprime to $\ch(K)$, and let~$E_2$ be the complementary elliptic curve. Let \mbox{$\isom\colon E_1[\degree]\stackrel{\sim}{\longrightarrow}E_2[\degree]$} be the induced canonical isomorphism (with respect to an embedding of $C$). Then~$\isom$ inverts the Weil pairing, i.e.
\begin{equation}
e_\degree(P, Q)=e_\degree\left(\isom(P),\isom(Q)\right)^{-1}
\label{eq:isotropy_condition}
\end{equation}
for any $P,Q\in E_1[\degree](\kbar)$.
\end{cor}

In other words, Lemma~\ref{lemma:mumford} provides a criterion for deciding when a polarization ``descends'' through an isogeny. In view of Lemma~\ref{lemma:mumford}, our starting data are two elliptic curves $E_1, E_2$ and an isomorphism $\isom\colon~E_1[\degree]\stackrel{\sim}{\longrightarrow}~E_2[\degree]$ that is \emph{anti-symplectic} with respect to the Weil pairing, i.e. $\isom$ satisfies~\eqref{eq:isotropy_condition} for any $P,Q\in E_1[\degree](\kbar)$, where~$\degree$ is coprime to~$\ch(K)$.
\par
We assume the usual principal polarization on $ E_1\times E_2$, given by the divisor
$$
\Theta=\{\zero1\}\times E_2+E_1\times\{\zero2\}.
$$
Let $\Gamma_\isom\subset(E_1\times E_2)[\degree]$ denote the graph of~$\isom$ and let
$$
\isog\colon E_1\times E_2\to(E_1\times E_2)/\Gamma_\isom=:J
$$
be the canonical map, which is clearly an isogeny.
\begin{lemma}
\label{lemma:Jpolarized}
The isogeny $\isog\colon E_1\times E_2\to J$ induces a principal polarization of $J$.
\end{lemma}
\begin{proof}
See p.~156 in~\cite{freykani}.
\end{proof}

\begin{lemma}
\label{lemma:uniqueC}
Suppose that $\degree$ is odd. Then there exists a unique effective divisor $C$ on $J\otimes\kbar$ such that $D=\isog^*(C)$ is linearly equivalent to $\degree\Theta$ and fixed by $[\minus 1]_{\E1\times\E2}$. The divisor $C$ is fixed by $[\minus 1]_J$ and principally polarizes $J$.
\end{lemma}
\begin{proof}
See Proposition 1.1 and Corollary 1.2 in \cite{freykani}.
\end{proof}

It is a well known theorem of Weil (Satz 2 in~\cite{weil}) that over $\kbar$ any principally polarized abelian surface is either a Jacobian or a product of two elliptic curves (with the usual polarizations). Therefore the question of whether or not $J$ is a Jacobian reduces to the question of whether or not the divisor $C$ is irreducible.
\par
\begin{lemma}
The divisor $C$ is irreducible if and only if the divisor~$D$ is irreducible.
\end{lemma}
\begin{proof}
See Proposition 1.3 in \cite{freykani}.
\end{proof}
\par
\begin{remark}
\label{rem:reducibleC}
There exist examples with $C$ reducible. Let $\gamma\colon E_1\to E_2$ be an isogeny of degree~\mbox{$\degree-1$} and let $\isom\colon E_1[\degree]\stackrel{\sim}{\lra} E_2[\degree]$ be the anti-symplectic isomorphism that is the restriction of $\gamma$ to the $\degree$-torsion. Then $C$ is reducible. Moreover, whenever~$C$ is reducible, $E_1$, $E_2$, and the irreducible components of $C$ are all isogenous and therefore $J$ is generically a Jacobian.
\end{remark}

\section{Gluing two elliptic curves along the $3$-torsion}
\label{section:gluing33}
In this section we will deal with the case $\degree=3$, given two elliptic curves from the Hesse pencil. We begin by fixing additional assumptions.
\begin{assumption}
The field $K$ is of characteristic $\ch(K)\neq 2,3$ and it contains a primitive third root of unity.
\end{assumption}
From now on, let $\w\in K$ be such that $1+\w+\w^2=0$.
\par
\subsection{Prerequisites}
The one-dimensional family of curves given by
\begin{equation}
\label{eq:hesse_pencil}
E_a\colon x^3+y^3+z^3+3a xyz=0,
\end{equation}
is called the \emph{Hesse pencil}.\footnote{The $3$ does not usually appear in the definition. We added it for simplicity.} With the exception \mbox{of $a^3=-1$,} each $a\in K$ defines an elliptic curve $E_a$. We choose $[-1:1:0]$ to be the identity element. The group morphisms are given in the Appendix. We denote the set of elliptic curves in the Hesse pencil by~$\hesse$. The $j$-invariant of $E_a$ is
\begin{equation}
j({E_a})=-\frac{27a^3(a^3-8)^3}{(a^3+1)^3}
\label{eq:j_lambda}
\end{equation}
and $j\colon\hesse\to\Aff^1$ is $12$-to-$1$, except above $j=0$ \mbox{and $j=1728$}. In particular, the elements of 
\begin{align*}
\Biggl\{ & a, a\w, a\w^2, \frac{2-a}{1+a}, \frac{2-a}{1+a}\w, \frac{2-a}{1+a}\w^2, \frac{2\w-a}{\w+a},\\
 &\frac{2\w-a}{\w+a}\w, \frac{2\w-a}{\w+a}\w^2, \frac{2\w^2-a}{\w^2+a}, \frac{2\w^2-a}{\w^2+a}\w, \frac{2\w^2-a}{\w^2+a}\w^2 \Biggr\}
\end{align*}
define isomorphic elliptic curves.
\par
The $3$-torsion subgroup of every elliptic curve in~$\hesse$ is fully $K$-rational and is given by $xyz=0$. Therefore the same nine points in~$\PP^2$ are the $3$-torsion points of every element of~$\hesse$ and each of the points can be given by homogeneous coordinates that are a permutation of $\{0,1,u\}$, where $u^3=-1$. Moreover, the Hesse pencil is exactly the family of all cubics passing through these nine points. \par
There is a partial converse in the form of Lemma~\ref{lemma:converse} below. A short direct proof is included in the Appendix (see~\S\ref{app:proof}).
\begin{lemma}
\label{lemma:converse}
Every elliptic curve over $K$ with fully $K$-rational $3$-torsion is isomorphic to a quadratic twist of an element of the Hesse pencil.
\end{lemma}
This can be thought of as an explicit realization of the fact that the modular curve~$X(3)$ has genus zero (see pp. 22--23 in~\cite{shimura}, for example).
\par
Now let $S=[-1:0:1]$ and $T=[-w:1:0]$. These two points generate $E[3]$ for every $E\in\hesse$. From now on, we are going to fix an isomorphism $\eta\colon E[3]\stackrel{\sim}{\lra}(\ZZ/3\ZZ)^2$ for every elliptic curve $E$ in~$\hesse$, given by $\eta(S)=(1,0)$ and $\eta(T)=(0,1)$.
\par
Let $E$ be an elliptic curve in the Hesse pencil. The Weil pairing on $E[3]$ is completely determined by the value $e_3(S,T)$ and we can easily calculate $e_3(S,T)=\w.$ 
For example, using Chapter III~\S8 in~\cite{aec}, we find that $e_3(S,T)=g(P+T)/g(P)$, where
$$
g = \frac{x^2z+y^2x+z^2y}{xyz}\in K(E)
$$
and $P\in E(\kbar)\setminus\left(E[3]\cup(P+E[3])\right)$. It follows that the Weil pairing on $E[3]$ is given by
$$
e_3(P,Q)=\w^{\det(\eta(P),\eta(Q))}
$$
and we can interpret it as the determinant map
$$
\det\colon\ZZ/3\ZZ\times\ZZ/3\ZZ\to\ZZ/3\ZZ.
$$
Since $\Aut(\ZZ/3\ZZ)\cong \GL_2(\ZZ/3\ZZ)$ is a group of order~$48$, an anti-symplectic isomorphism $E_1[3]\stackrel{\sim}\lra E_2[3]$ corresponds to one of the $24$ elements of the coset $\smallmat{0}{1}{1}{0} \SL_2(\ZZ/3\ZZ)$. However, since each isomorphism can be composed with $[\minus 1]$, we are left with $12$ distinct cases at most.
\par
Before we deal with the general case, we consider a concrete example in which gluing two elliptic curves does not give a Jacobian.
\begin{example}
\label{ex:split_glue}
Let $a=-(1+2t^3)/(3t^2)$ for some~$t\in K$, such that $$t(t^3-1)(8t^3+1)\neq 0.$$ Then the elliptic curve
$$E_a\colon x^3+y^3+z^3-\frac{1 + 2 t^3}{t^2} x y z =0$$ 
has a rational point $[t:t:1]$ of order two. Let $b=(1-4t^3)/(3t)$. Then the map
$$
\gamma\colon E_a\to E_b,\quad [x:y:z]\mapsto[f_1(x,y,z) : f_2(x,y,z) : f_3(x,y,z)],
$$
where
\begin{align*}
&f_1 = x (-2 t^2 y^2 - t^2 x y + t^2 x^2 - y z + 2 t^3 x z + t z^2),\\
&f_2 = y (-2 t^2 x^2 - t^2 x y + t^2 y^2 - x z + 2 t^3 y z + t z^2),\\
&f_3 = t z (x + y + t z) (x + y - 2 t z),
\end{align*}
is an isogeny whose kernel is the cyclic group of order two that is generated by the point $[t:t:1]$. Restricting $\gamma$ to the $3$-torsion, we obtain the isomorphism \mbox{$\isom\colon E_a[3]\to E_b[3]$} that corresponds to $\smallmat{1}{0}{0}{2}\in\GL_2(\ZZ/3\ZZ)$. We have that \mbox{$J:=(E_a\times E_b)/\Gamma_\isom$} is isomorphic to $E_a\times E_b$ as a principally polarized abelian surface (recall Remark~\ref{rem:reducibleC}). We note that $a$ and $b$ satisfy
\begin{equation}
3a^2b^2 + a^3 + b^3 - 3ab + 2 = 0,
\label{eq:disc_vanishes}
\end{equation}
which is an equation describing a singular affine curve of genus zero.
\end{example}

We now consider the isomorphism from Example~\ref{ex:split_glue} in full generality. From now on, we fix $\isom=\smallmat{1}{0}{0}{2}$. Let $E_1$ and $E_2$ be two elliptic curves in $\hesse$, corresponding to parameters $a$ and $b$, respectively. Let $A$ and $G$ respectively denote the images of~$E_1\times E_2$ and $\Gamma_\isom$ in $\PP^8$ under the Segre embedding $\sigma$, given by
\begin{equation*}
\left([x_1:y_1:z_1],[x_2:y_2:z_2]\right)\mapsto [x_1x_2:x_1y_2:x_1z_2:y_1x_2:y_1y_2:y_1z_2:z_1x_2:z_1y_2:z_1z_2].
\end{equation*}
The identity element of $A$ is $O_A=[1:-1:0:-1:1:0:0:0:0]$ and the inversion morphism $[\minus 1]_A$ is given by
\begin{equation}
[X_1:X_2:\cdots:X_9]\mapsto[X_5:X_4:X_6:X_2:X_1:X_3:X_8:X_7:X_9].
\label{eq:linear_minus_id}
\end{equation}
Let $\Theta=\sigma(E_1\times\{O_2\})+\sigma(\{O_1\}\times E_2)$ and let $D$ denote the effective divisor on~$A$ that is linearly equivalent to~$3\Theta$, invariant under $[\minus 1]_A$, and invariant under the translation by the points of~$G$. Let $\isog\colon A\to J$ denote the isogeny with kernel $G$ and let $C=\isog(D)$.

\subsection{The computations}
Now we will go over the steps that lead to the Igusa--Clebsch invariants of a genus-$2$ curve whose Jacobian is isomorphic to $(E_1\times E_2)/\Gamma_\isom$ as a principally polarized abelian surface, if such a curve exists. All computations were performed using the computational algebra system \textsc{Magma}~\cite{magma}. Naturally, some technical details will be omitted.
\par
The \emph{Igusa--Clebsch invariants} are the invariants $A'$, $B'$, $C'$, $D'$ defined on p.~319 of~\cite{mestre}. The \emph{Igusa invariants} are the invariants $J_2$, $J_4$, $J_6$, $J_8$, $J_{10}$ defined on p.~324 ibid.
By the corresponding \emph{absolute invariants} we mean the values
$$
j_1 = \frac{J_2^5}{J_{10}},\quad j_2 = \frac{J_2^3 J_4}{J_{10}},\quad j_3=\frac{J_2^2 J_6}{J_{10}}.
$$
\par
We start by computing the ideal $I=I(A)$ that defines $A$ as a variety in~$\PP^8$. This is a straightforward computation and is omitted here.
\begin{lemma}
\label{lemma:lemma1}
Let $\mathcal{W}_1$ and $\mathcal{W}_2$ denote the set of geometric points of order two on~$\sigma(E_1\times\{O_2\})$ and the set of geometric points of order two on $\sigma(\{O_1\}\times E_2)$, respectively. Then any hyperplane section on $A$ that is invariant under~$[\minus 1]_A$ contains either $\mathcal{W}_1\cup \mathcal{W}_2$ or its complement in $A[2](\kbar)$.
\end{lemma}
\begin{proof}
The two eigenspaces of~\eqref{eq:linear_minus_id} are respectively generated by the sets
\begin{equation}
\begin{split}
S_1&=\{X_1+X_5, X_2+X_4, X_3+X_6, X_7+X_8, X_9\},\\
S_2&=\{X_1-X_5, X_2-X_4, X_3-X_6, X_7-X_8\}.
\end{split}
\label{eq:invol_invariants}
\end{equation}
By adding the corresponding linear forms from~\eqref{eq:invol_invariants} to~$I$, we find that~$A[2](\kbar)$ consists of six points that are in the zero locus of the ideal generated by~$S_1$ and ten points that are in the zero locus of the ideal generated by~$S_2$. Since every linear form that is an eigenvector for~$[\minus 1]_A$ is a linear combination of the elements of exactly one of these two sets, the claim follows.
\end{proof}
\begin{corollary}
\label{cor:cor1}
The quotient $J=A/G$ is a Jacobian if and only if the six geometric points of $\mathcal{W}_1\cup \mathcal{W}_2$ are the only $2$-torsion points on the divisor~$D$.
\end{corollary}
\begin{proof}
The divisor $C$ is either a curve of genus two or a sum of two elliptic curves that meet in a rational $2$-torsion point. Since $[\minus 1]_J$ induces a hyperelliptic involution~$\invol$ on the irreducible components of~$C$, we conclude that $C(\kbar)$ contains exactly six points fixed by~$\invol$ if and only if it is irreducible and that it contains exactly seven points fixed by~$\invol$ if and only if it is reducible. Since $\deg\isog$ is odd, the restriction of $\isog$ to the $2$-torsion is an isomorphism and there is exactly one geometric point of $(E_1\times E_2)[2]$ above each point of $C(\kbar)$ that is fixed by~$\invol$. Therefore $D(\kbar)$ contains at most seven $2$-torsion points. By Lemma~\ref{lemma:lemma1},~$D(\kbar)$ contains at least the order-$2$ points of $\sigma(E_1\times\{\zero2\})$ and~$\sigma(\{\zero1\}\times E_2)$ and the claim follows.
\end{proof}
A fact crucial to our approach is that the translations by the points of~$A[3]$ are linear. In fact, they can be extended to automorphisms of $\PP^8$. This is a consequence of the fact that $A$ is embedded into $\PP^8$ via an embedding corresponding to $L(3\Theta)$. It can also be shown directly, using the addition formulas. In particular, the group of translations by the points of~$G$ is generated by the following two automorphisms:
\begin{align*}
[X_1:X_2:\dots:X_9]&\mapsto[X_5:X_6:X_4:X_8:X_9:X_7:X_2:X_3:X_1]:\\
[X_1:X_2:\dots:X_9]&\mapsto[X_1:\w X_2:\w^2 X_3:\w^2 X_4:X_5:\w X_6:\w X_7:\w^2 X_8:X_9]
.\end{align*}
From this we immediately determine that the nine effective divisors invariant under the action of~$G$ and linearly equivalent to~$3\Theta$ are the hyperplane sections defined by the following linear forms:
\begin{align*}
L_1 &= X_1 + X_5 + X_9,&
L_6 &= \w^2 X_3 + \w X_4 + X_8,\\
L_2 &= \w X_1 + \w^2 X_5 + X_9,&
L_7 &= \w X_2 + \w^2 X_6 + X_7,\\
L_3 &= \w^2 X_1 + \w X_5 + X_9,&
L_8 &= \w X_3 + \w^2 X_4 + X_8,\\
L_4 &= X_3 + X_4 + X_8,&
L_9 &= \w^2 X_2 + \w X_6 + X_7.\\
L_5 &= X_2 + X_6 + X_7,&
\end{align*}
We note that the divisor~$D$, that is invariant under~$[\minus 1]_A$, is defined by~$L_1=0$ and does not contain~$O_A$. Now we can compute the scheme that is the intersection of~$D$ and the nine points of~$A[2](\kbar)$ that are not $2$-torsion points on $\sigma(E_1\times\{O_2\})$ or~$\sigma(\{O_1\}\times E_2)$ and apply Corollary~\ref{cor:cor1}. Dehomogenizing by setting $X_9=1$, taking the corresponding ideal in the ring~$K[X_1,\dots,X_8,a,b]$, and eliminating the~$X_i$ gives
\begin{equation}
3a^2b^2 + a^3 + b^3 - 3ab + 2 = 0
\label{eq:disc_vanishes2}
\end{equation}
Note that this matches \eqref{eq:disc_vanishes}. Thus we have obtained the following (cf. Proposition~2.2. and Corollary~2.3 in \cite{freykani}).
\begin{proposition}
The principally polarized abelian surface $J=A/G$ is a product of two elliptic curves if and only if~\eqref{eq:disc_vanishes2} holds, i.e. if and only if~$E_1$ and~$E_2$ are $2$-isogenous.
\end{proposition}
\par
Equation~\eqref{eq:disc_vanishes2} can be thought of as the analogue of~$\Phi_2(j(E_1),j(E_2))=0$ that is specific to our choice of~$\isom$. Here~$\Phi_2$ denotes the classical modular polynomial
\begin{align*}
\Phi_2(X,Y) = &\:X^3 + Y^3 - X^2Y^2 + 1488(X^2Y+XY^2) - 162000(X^2+Y^2)\:+\\
& + 40773375XY + 8748000000(X+Y) - 157464000000000.
\end{align*}
\par
The abelian surface $J$ can be found explicitly as the quotient of the variety~$A$ under the group action of~$G$, where $G$ acts by point translation (see Lecture~10 in~\cite{harris}, for example). Since~$\isog^*$ is injective and $\dim_K(L(nC))=n^2$ for every $n\in\NN$, we have that the subspace~$L(nD)^G=\isog^*(L(nC))$ of~$G$-invariants of~$L(nD)$ is of dimension~$n^2$ for every~$n\in\NN$. In particular, since $\LL(3C)$ is very ample, by finding nine linearly independent $G$-invariant elements of~$L(3D)$, we can obtain~$\isog$ as a map to~$\PP^8$. We may take the nine $G$-invariant forms~$L_i^3$ for this purpose. Unsurprisingly, explicitly computing~$\isog(A)$ is not feasible.
\par
From now on, let us assume that $3a^2b^2 + a^3 + b^3 - 3ab + 2 \neq 0$ so that~$C$ is irreducible. It follows that the global sections of~$\LL(2C)$ define the canonical map~$\kappa\colon J\to \KK$, where~$\KK=J/[\minus 1]\subset\PP^3$ is a Kummer surface (Proposition~4.23 in~\cite{16-6}). Therefore the four-dimensional $G$-invariant subspace $L(2D)^G\subset L(2D)$ defines the composition~$\psi=\kappa\comp\isog$. We have that~$\psi(D)$ is a conic in~$\PP^3$ and the image under~$\psi$ of the~\mbox{$2$-torsion} points that lie on~$D$ gives six pairwise distinct (geometric) points on~$\psi(D)$ that are the branch locus of the canonical $2$-to-$1$ map $C\to\psi(D)$. By finding a~$K$-rational point on the conic~$\psi(D)$, we obtain an isomorphism~$\psi(D)\stackrel{\sim}{\lra}\PP^1$ and the image of the six branch points gives us a sextic that defines a plane model of a hyperelliptic curve that is in the isomorphism class of~$C$. We can then directly compute the absolute invariants from this model. We may take the following four~$G$-invariant forms to define~$\psi$:
\begin{align*}
X_2X_4 + X_3X_7 + X_6X_8,
&\quad
X_2X_3 + X_4X_6 + X_7X_8,\\
X_2X_8 + X_3X_6 + X_4X_7,
&\quad
X_1^2 + X_5^2 + X_9^2.
\end{align*}
An alternative approach is to compute the curve~$C=\isog(D)$ directly, compute the canonical divisor~$K_C$, and then find the image in~$\PP^1$ of the six points of~$J[2](\kbar)$ that lie on~$C$, under the canonical map defined by~$L(K_C)$. However, this is significantly slower in practice than the Kummer surface approach.
\par
We make an important observation. The absolute invariants of~$C$, as functions of the parameters~$a$ and~$b$, will have certain symmetries. For example, the abelian surface $E_1\times E_2$ is isomorphic to~$E_2\times E_1$ and the isomorphism (which is just a permutation of the coordinates) leaves~$G$ intact so that \mbox{$(E_1\times E_2)/\Gamma_\isom$} and \mbox{$(E_2\times E_1)/\Gamma_\isom$} will give the same absolute invariants. Similarly, the same invariants are obtained if one starts with a pair $(E_1,E_2)\in\hesse^2$ defined by parameters~$(a\w,b\w^2)$ or~$(a\w^2,b\w)$.
\begin{remark}
Recall that for each curve in~$\hesse$ there are eleven other curves in~$\hesse$ that are isomorphic to it, with the exception of the usual two isomorphism classes. It is a natural question to ask which of the $144$ possible isomorphic pairs~$E_1\times E_2$ result in the same isomorphism class of~$C$ when taking the quotient by~$G$. It turns out that the pairs are partitioned into twelve sets of twelve pairs and every pair in the same set gives the same isomorphism class of~$C$. Moreover, all of the remaining eleven choices of~$\isom$ can be reduced to the case we are considering by applying an isomorphism to a suitable product of two elliptic curves in~$\hesse$. Our choice of~$\isom$ is particularly suitable for computations because of the simplicity of the equations defining~$D$ and the equation~\eqref{eq:disc_vanishes2}.
\end{remark}
To obtain the absolute invariants of~$C$ as functions of $(a,b)$, the first thing we do is make several degree estimates. For example, we can take~$a$ and~$b$ to be two large integers of comparable height, such as two large consecutive primes. We can also take~$a$ to be a large integer and~$b\in\{0,1\}$. This gives us estimates for the degrees of particular monomials that appear in the invariants. Then we notice that the so-called discriminant~$J_{10}$, that appears in the denominators, is going to be zero for choices of~$(a,b)$ that either do not define a pair of elliptic curves or do not define a quotient~$J$ that is a Jacobian. By factoring the invariants obtained for various choices of~$a,b\in\ZZ$ and combining this information with the degree estimates, we conclude that, up to a constant, $J_{10}$ equals
\begin{equation}
\label{eq:probable_j10}
(a^3+1)(b^3+1)(3a^2b^2 + a^3 + b^3 - 3ab + 2)^{12}.
\end{equation}
\par
To obtain the numerators, we use interpolation. We compute the absolute invariants of~$C$ for many choices of~$(a,b)$ and multiply by~\eqref{eq:probable_j10} in each case. We conclude from the aforementioned symmetries that the numerators are also linear combinations of monomials~$a^mb^n$, where $m\equiv n$ (mod 3). This significantly reduces the number of non-zero coefficients and makes the computation reasonably fast. Using our empirical bounds on the degrees and the coefficients, we interpolate over finite fields~$\FF_p$ for a suitable set of primes~$p$ and lift the results using the Chinese remainder theorem. Finally, we obtain the Igusa--Clebsch invariants from the Igusa invariants using the formulas in~\cite{mestre}. We summarize our results in the following proposition.
\begin{proposition}
\label{prop:prop2}
Let $E_1$ and $E_2$ be two elliptic curves over $K$, respectively given by the Hesse models 
\begin{align*}
E_1\colon x^3+y^3+z^3+3axyz&=0,\\
E_2\colon x^3+y^3+z^3+3bxyz&=0,
\end{align*}
with identity point $O=[-1:1:0]$, where $a,b\in K$ and $3a^2b^2 + a^3 + b^3 - 3ab + 2\neq 0.$ Let $\isom\colon E_1[3]\stackrel{\sim}{\lra}E_2[3]$ be the isomorphism defined by
$$[-1:0:1] \mapsto [-1:0:1],\quad[-\w:1:0] \mapsto[-\w^2:1:0],$$
and let~$\Gamma_\isom$ denote its graph. Then the principally polarized abelian surface \mbox{$(E_1\times E_2)/\Gamma_\isom$} is isomorphic to the Jacobian of a curve of genus two whose Igusa--Clebsch invariants are as follows:
\allowdisplaybreaks
\begin{align*}
I_2  =\;& 72 (9 a^6 b^6 - 30 (a^7 b^4 + a^4 b^7) - 88 a^5 b^5 + a^8 b^2 + a^2 b^8 + 54 (a^6 b^3 + a^3 b^6) + 65 a^4 b^4+\\
& - 32 (a^7 b + a b^7) - 104 (a^5 b^2 + a^2 b^5) + 40 (a^6 + b^6) +  44 a^3 b^3 + 100 (a^4 b + a b^4)+\\
& - 68 a^2 b^2 + 16 (a^3 + b^3) + 112 a b - 20),\\[.5em]
I_4  =\;& 36 (3 a^2 b^2 + a^3 + b^3 - 3 a b + 2)^4 (9 a^4 b^4 + 240 a^3 b^3 + 8 (a^4 b + a b^4) + 240 a^2 b^2+\\
&+ 160 (a^3 + b^3) + 256 a b + 320),\\[.5em]
I_6  =\;& 72 (3 a^2 b^2 + a^3 + b^3 - 3 a b + 2)^4 (729 a^{10} b^{10} - 3402 (a^{11} b^8 + a^8 b^{11})+\\
&+ 30456 a^9 b^9 +  81 (a^{12} b^6 + a^6 b^{12}) - 70794 (a^{10} b^7 + a^7 b^{10}) - 
 201555 a^8 b^8 + \\
&- 2160 (a^{11} b^5 + a^5 b^{11}) + 60 (a^{12} b^3 + a^3 b^{12}) + 106560 (a^9 b^6 + a^6 b^9) +\\
&-  148932 a^7 b^7 - 121608 (a^{10} b^4 + a^4 b^{10}) + 
 480 (a^{11} b^2 + a^2 b^{11}) +\\
&- 358740 (a^8 b^5 + a^5 b^8)  - 8 (a^{12} + b^{12}) + 156928 (a^9 b^3 + a^3 b^9) +\\
&+ 336444 a^6 b^6 -  50160 (a^{10} b + a b^{10}) + 81072 (a^7 b^4 + a^4 b^7) +\\
&- 462096 a^5 b^5 - 167112 (a^8 b^2 + a^2 b^8) + 84224 (a^9 + b^9) +\\
&+  455568 (a^6 b^3 + a^3 b^6) + 761040 a^4 b^4 + 181152 (a^7 b + a b^7) +\\
&- 93600 (a^5 b^2 + a^2 b^5) +  219552 (a^6 + b^6) + 383424 a^3 b^3 +\\
&+ 564480 (a^4 b + a b^4) + 88512 a^2 b^2 + 74624 (a^3 + b^3) + 314112 a b - 55040),\\[.5em]
I_{10} =\;& 36864 (a^3+1) (b^3+1) (3 a^2 b^2 + a^3 + b^3 - 3 a b + 2)^{12}.
\end{align*}
\end{proposition}
\par

When $K$ is a number field or a finite field of characteristic $\ch(K)>5$, it is possible to construct a genus-$2$ curve over~$K$ with given Igusa--Clebsch invariants (see~\cite{cardona} and~\cite{mestre}). When~$K$ is a number field, the recent work of Bruin, Sijsling, and Zotine~\cite{sijsling} allows one to verify numerically over~$\CC$ that a curve obtained from Igusa--Clebsch invariants of the form above indeed has a $(3,3)$-split Jacobian.
\par

\begin{ack} The author wishes to thank Robin de Jong and Fabien Pazuki for the fruitful discussions that inspired this paper. The author is also grateful to Ronald van Luijk, Christophe Ritzenthaler, and Jeroen Sijsling for their helpful comments and remarks. Portions of the paper have already appeared in the author's Ph.D. thesis, completed at the universities of Leiden and Bordeaux. Most of the work presented was completed during the author's stay at the University of Ulm.
\end{ack}
\setcounter{section}{0}
\setcounter{equation}{0}
\renewcommand{\theequation}{\thesection.\arabic{equation}}
\setcounter{figure}{0}
\setcounter{table}{0}

\appendix
\section*{Appendix}
\label{appendix}
\renewcommand{\thesection}{A}
\subsection{Hesse pencil morphisms}
\label{app:morphisms}
Given a smooth projective plane curve over~$K$ of the form
\begin{equation}
x^3+y^3+z^3+\lambda xyz=0
\label{eq:hesse_eqn_appendix}
\end{equation}
for $\lambda\in K$, we give it the structure of an elliptic curve as follows. The identity element is $O=[1:-1:0]$. The inversion morphism is given by
$$
[x:y:z]\mapsto[y:x:z].
$$
The addition morphism is given by
$$
\left([x_1:y_1:z_1],[x_2:y_2:z_2]\right)\mapsto[
y_1^2 x_2 z_2 - y_2^2 x_1 z_1 :
x_1^2 y_2 z_2-x_2^2 y_1 z_1 :
z_1^2 x_2 y_2  - z_2^2 x_1 y_1].
$$
The point duplication morphism is given by
$$
[x:y:z]\mapsto[y (x^3 - z^3) : x (z^3 - y^3) : z (y^3 - x^3)].
$$
One can easily obtain the formulas for the corresponding morphisms on the product of two curves of the form~\eqref{eq:hesse_eqn_appendix}, embedded into $\PP^8$ via the Segre embedding.

\subsection{Examples}
\label{app:examples}
\setcounter{theorem}{0}
We will make use of the notations from Section~\ref{section:gluing33}.
\begin{example}
Let $a=b=0$ so that $A\cong E^2$, where $E$ is the Fermat curve~\mbox{$x^3+y^3+z^3=0$.} Then by Proposition~\ref{prop:prop2} we have that $A/G\cong\Jac(C)$, where $C$ is a genus-$2$ curve with Igusa--Clebsch invariants $[-90 : 720 : -15480 : 144]$. Applying Mestre's algorithm yields an affine curve
$$
C\colon dy^2 = (x^3+5)(4x^3+5),
$$
for some $d\in K$. We recognize this as a case of~\eqref{eq:generic_C}. The curve~$C$ admits maps
\begin{equation}
\label{eq:example1}
\begin{split}
\phi_1\colon (x,y)&\mapsto \left(
-\frac{15x^2}{x^3 + 5},\:
y\frac{5(x^3 - 10)}{(x^3 + 5)^2}
\right),\\
\phi_2\colon (x,y)&\mapsto \left(
-\frac{75x}{4x^3 + 5},\:
y\frac{25(8x^3 - 5)}{(4x^3 + 5)^2}
\right),
\end{split}
\end{equation}
whose images are elliptic curves, respectively defined by affine plane models
\begin{align*}
E_1\colon & dy^2 = x^3 + 100,\\
E_2\colon & dy^2 = x^3 + 625.
\end{align*}
The constants in~\eqref{eq:example1} are chosen for the simplicity of the models. As expected, we have $j(E)=j(E_1)=j(E_2)=0$. Moreover, $\phi_1$ ramifies at $(0,\pm 5)$ and $\phi_2$ ramifies at~$\pm\infty$. The ramification points lie above rational order-$3$ points of the corresponding elliptic curves in both cases; these are the points $(0,\pm 10)$ and~$(0,\pm 25)$, respectively. In particular, the ramification does not occur above $2$-torsion points.
\end{example}

\begin{example}[\S\ref{subsub:both_special} revisited]
Let $C$ be a genus-$2$ affine curve defined by $$dy^2=x(x^2 + 1)(4x^2 + 3)$$
for some $d\in K$. This curve admits $3$-to-$1$ coverings
\begin{align*}
\phi_1\colon C\to E_1,&\quad(x,y)\mapsto \left(
\frac{1}{x(4x^2 + 3)},\:
y\frac{4x^2 + 1}{x^2(4x^2 + 3)^2}
\right),\\
\phi_2\colon C\to E_2,&\quad(x,y)\mapsto \left(
\frac{4x^3}{x^2 + 1},\:
y\frac{4x(x^2 + 3)}{(x^2 + 1)^2}
\right),
\end{align*}
where the images are elliptic curves defined by
\begin{align*}
E_1\colon & dy^2 = x^3 + x,\\
E_2\colon & dy^2 = x^3 + 108x.
\end{align*}
We have $j(E_1)=j(E_2)=1728$. Moreover, $\infty$ is a triple ramification point for $\phi_1$ and~$(0,0)$ is a triple ramification point for~$\phi_2$. Both points lie above~$(0,0)$, which is a point of order two on both~$E_1$ and~$E_2$. Suppose that $\sqrt{3}\in K$ (i.e. $\sqrt{-1}\in K$, given that $\w\in K$) and let \mbox{$a=b=-1+\sqrt{3}$.} This parameter defines $E\in\hesse$ with $j(E)=1728$ and by Proposition~\ref{prop:prop2} we have that $E^2/G$ is principally polarized by a genus-$2$ curve whose Igusa--Clebsch invariants are $[774 : 9648 : 2763360 : 27648]$; these are easily verified to be the invariants of $C$.
\end{example}

\begin{example}[\S\ref{subsub:both_special} revisited]
\label{example:bad_conclusion}
Let $C$ be a genus-$2$ affine curve defined by $$dy^2=x(2x^2 + 4x + 3)(3x^2 + 4x + 2)$$
for some $d\in K$. Then $C$ has Igusa--Clebsch invariants $[86:13456:471968:6718464]$ and it admits $3$-to-$1$ coverings
\begin{align*}
\phi_1\colon C\to E_1,&\quad(x,y)\mapsto \left(
\frac{18x^3}{3 x^2 + 4 x + 2},\:
y\frac{18x(3 x^2 + 8 x + 6)}{(3 x^2 + 4 x + 2)^2}
\right),\\
\newoverline[.9]{\phi}_1\colon C\to E_1,&\quad(x,y)\mapsto \left(
\frac{18}{x (2 x^2 + 4 x + 3)},\:
y\frac{18(6x^2 + 8 x + 3)}{x^2 (2 x^2 + 4 x + 3)^2}
\right),
\end{align*}
where $E_1\colon dy^2 = x(x^2 + 44x + 486)$. We have $j(E_1)=-873722816/59049$. Note that $\newoverline[.9]{\phi}_1=\phi_1\comp \xi$, where $\xi\in\Aut(C)$ is given by $\xi(x,y)=(1/x,y/x^3)$. The maps ramify above $(0,0)$, which is a $2$-torsion point, and send disjoint sets of three \W points to $\infty$, i.e. the identity point. However, the two coverings are not complementary. The curve $C$ admits another pair of $3$-to-$1$ coverings, which are generic. For example, we can take
\begin{align*}
\phi_2\colon C\to E_2,&\quad(x,y)\mapsto \left(
\frac{2x^3 + 5x^2 + 4x - 2}{3x^2 + 4x + 2},\:
y\frac{2(x+2)(x^2+2)}{(3x^2 + 4x + 2)^2}
\right),\\
\newoverline[.9]{\phi}_2\colon C\to E_2,&\quad(x,y)\mapsto \left(
\frac{-2x^3 + 4x^2 + 5x + 2}{x (2 x^2 + 4 x + 3)},\:
y\frac{2(2x+1)(2x^2+1)}{x^2 (2 x^2 + 4 x + 3)^2}
\right),
\end{align*}
where $E_2\colon y^2 = x^3 - x^2 + x  + 3$ and $j(E_2)=64/9$. As before, we also have $\newoverline[.9]{\phi}_2=\phi_2\comp\xi$. The map~$\phi_2$ ramifies at the two points with $x=-2/3$, whereas $\newoverline[.9]{\phi}_2$ ramifies at the two points with $x=-3/2$. The ramification occurs above points of infinite order. By using Proposition~\ref{prop:prop2}, we can show that if $E$ and $E'$ are elliptic curves that can be glued along their $3$-torsion to give a Jacobian of a genus-$2$ curve in the isomorphism class of~$C$ and $j(E)=-873722816/59049$, then $j(E')=64/9$. Therefore~$\phi_1$ and~$\phi_2$ are complementary in the sense of Definition~\ref{def:complement}, whereas $\phi_1$ and $\newoverline[.9]{\phi}_1$ are not.
\end{example}

\subsection{General formulas for coverings of degree 3}
\label{app:explicit_covering}
More generally, recalling notations from \S~\ref{sub:generic} and omitting all conditions on the parameters $a,b,c,d\in K$, let
\begin{align*}
P(x)&=x^3+ax^2+bx+c, & Q(x)&=4cx^3+ b^2x^2 + 2bcx + c^2,\\
f_1(x)&=\frac{x^2}{P(x)}, & f_2(x)&=\frac{(b x + 3 c)^2 ((b^3 - 4 a b c + 9 c^2) x + c (b^2 - 3 a c))}{Q(x)}.
\end{align*}
Then the genus-$2$ affine curve $C$ defined by $dy^2=P(x)Q(x)$ admits degree-$3$ coverings
\begin{equation}
\label{eq:covering_formula}
\begin{split}
\phi_1\colon C\to E_1,&\quad(x,y) \mapsto \left(
\mapd1(x),\:\frac{y}{x}\mapd1'(x)
\right),\\
\phi_2\colon C\to E_2,&\quad(x,y) \mapsto \left(
\mapd2(x),\:\frac{y}{b x + 3 c}\mapd2'(x)
\right),
\end{split}
\end{equation}
where $\mapdi'(x)=\frac{\d}{\d x}\mapdi(x)$. The elliptic curves $E_1$ and $E_2$ are defined by
\begin{align*}
E_1\colon & \frac{d}{\Delta_1}y^2 = x^3 + \frac{2(-a b^2 + 6 a^2 c - 9 b c)}{\Delta_1} x^2 + \frac{(b^2 - 12 a c)}{\Delta_1} x + \frac{4 c}{\Delta_1},\\[.5em]
E_2\colon & \Delta_2dy^2 = x^3 + (a b^3 - 27 b^2 c + 54 a c^2) x^2 + (b^7 - 18 a b^5 c + 54 a^2 b^3 c^2 +\\
& +189 b^4 c^2 - 972 a b^2 c^3 + 729 a^2 c^4 + 729 b c^4) x - c (2 b^3 - 9 a b c + 27 c^2)^3,
\end{align*}
where
\begin{align*}
\Delta_1 &= a^2 b^2 - 4 b^3 - 4 a^3 c + 18 a b c - 27 c^2,\\
\Delta_2 &= b^3 - 27 c^2.
\end{align*}
Suppose instead that, as in~\S\ref{subsub:second_special}, we have
\begin{align*}
P(x)&=(bx+3c)(9cx^2+2b^2x+3bc), & Q(x)&=4cx^3+ b^2x^2 + 2bcx + c^2,\\
f_1(x)&=\frac{x^2}{P(x)}, & f_2(x)&=\frac{(b x + 3 c)^3}{Q(x)}.
\end{align*}
Then the genus-$2$ affine curve $C$ defined by $dy^2=P(x)Q(x)$ admits degree-$3$ coverings defined by the formulas in~\eqref{eq:covering_formula}, where
\begin{align*}
E_1\colon & \frac{9bd}{4\Delta^3}y^2 = x^3 + \frac{3}{\Delta} x^2 - \frac{3(5b^3 + 108c^2)}{4\Delta^3} x + \frac{1}{\Delta^3},\\[.5em]
E_2\colon & \frac{d}{c}y^2 = x^3 + \frac{2\Delta}{c}x^2 - 27\Delta x,
\end{align*}
and $\Delta=b^3 - 27c^2$.
\subsection{Proof of Lemma~\ref{lemma:converse}}
\label{app:proof}
\begin{proof}
Suppose that $E$ is defined by the \W equation $$F(x,y,z)=-y^2z + x^3 + axz^2 + bz^3=0.$$ Then the Hessian of $F$ is given by $H(x,y,z)=3xy^2 + 3ax^2z + 9bxz^2 - a^2z^3$. The intersection of $E$ and the curve defined by $H(x,y,z)=0$ consists of the nine inflection points of $E$, that are all $K$-rational by assumption. Computing the intersection explicitly gives us the kernel polynomial
\begin{equation}
\label{eq:one}
3x^4 + 6ax^2 + 12bx - a^2,
\end{equation}
which must split completely over~$K$. Suppose that the (necessarily pairwise distinct) roots of~\eqref{eq:one} are $t_1,t_2,t_3,t_4\in K$. Expanding $(x-t_1)\cdots(x-t_4)$ and equating with~\eqref{eq:one} gives
\begin{align*}
\label{eq:two}
\begin{split}
t_4 &= -t_1-t_2-t_3,\\
-2a &= t_1^2 + t_1t_2 + t_2^2 + t_1t_3 + t_2t_3 + t_3^2,\\
4b &= (t_1 + t_2)(t_1 + t_3)(t_2 + t_3),\\
a^2 &= 3t_1t_2t_3(t_1 + t_2 + t_3).
\end{split}
\end{align*}
Eliminating $a$ and $b$, we obtain that $(t_1,t_2,t_3)$ lies on the union of the surface
\begin{equation}
\label{eq:three}
t_1^2  + \w t_2^2 + \w^2 t_3^2 - 2\w^2 t_1t_2 - 2\w t_1t_3 - 2t_2t_3 = 0
\end{equation}
and its Galois conjugate. By renaming the roots if necessary, we can and do assume that~\eqref{eq:three} holds. Now let
$$
t = \frac{3 t_1 + (5 + \w) t_2 + (4 - \w) t_3}{(1 + 2 \w) (t_2 - t_3)},\quad
u =\frac{12(t_1 + (2 + \w)t_2 + (1 - \w)t_3 )}{(t_2 - t_3)^2}.
$$
Then $t^3\neq -1$ and
\begin{align*}
au^2 &= -3t(t^3 - 8),\\
bu^3 &= -2(t^6 + 20 t^3 - 8).
\end{align*}
Finally, the elliptic curve defined by the \W equation
$$
-y^2z + x^3 -3t(t^3 - 8)xz^2 -2(t^6 + 20 t^3 - 8)z^3=0
$$
is isomorphic to the element of $\hesse$ defined by $$x^3+y^3+z^3+3txyz=0$$ via the isomorphism
$$
[x:y:z]\mapsto[
3 t x - (1 + 2 \w) y + 3 (t^3 + 4) z:
3 t x + (1 + 2 \w) y + 3 (t^3 + 4) z:
6(x - 3t^2z)].
$$
\end{proof}
\bibliographystyle{plain}
\bibliography{ms}

\begin{thebibliography}{10}

\bibitem{bolza}
O.~Bolza.
\newblock Zur {R}eduction hyperelliptischer {I}ntegrale erster {O}rdnung auf
  elliptische mittels einer {T}ransformation dritten {G}rades.
\newblock {\em Math. Ann.}, 50(2-3):314--324, 1898.

\bibitem{magma}
W.~Bosma, J.J. Cannon, and C.~Playoust.
\newblock The {M}agma algebra system {I}.
\newblock {\em J. Symbolic Comput.}, 24:235--265, 1997.
\newblock (Magma's homepage is at
  \href{http://magma.maths.usyd.edu.au/magma}{\texttt{http://magma.maths.usyd.edu.au/magma}}).

\bibitem{brioschi}
F.~Brioschi.
\newblock Sur la r\'{e}duction de l'int\'{e}grale hyperelliptique \`a
  l'elliptique par une transformation du troisi\`eme degr\'{e}.
\newblock {\em Ann. Sci. \'{E}cole Norm. Sup. (3)}, 8:227--230, 1891.

\bibitem{howe}
R.~Br\"{o}ker, E.~W. Howe, K.~E. Lauter, and P.~Stevenhagen.
\newblock Genus-2 curves and {J}acobians with a given number of points.
\newblock {\em LMS J. Comput. Math.}, 18(1):170--197, 2015.

\bibitem{bruin}
N.~Bruin and K.~Doerksen.
\newblock The arithmetic of genus two curves with {$(4,4)$}-split {J}acobians.
\newblock {\em Canad. J. Math.}, 63(5):992--1024, 2011.

\bibitem{sijsling}
N.~Bruin, J.~Sijsling, and A.~Zotine.
\newblock Numerical computation of endomorphism rings.
\newblock In {\em ANTS XIII: Proceedings of the Thirteenth Algorithmic Number
  Theory Symposium}, volume~2 of {\em Open Book Series}, pages 155--171, 2018.

\bibitem{cardona}
G.~Cardona and J.~Quer.
\newblock Field of moduli and field of definition for curves of genus 2.
\newblock In {\em Computational aspects of algebraic curves}, volume~13 of {\em
  Lecture Notes Ser. Comput.}, pages 71--83. World Sci. Publ., Hackensack, NJ,
  2005.

\bibitem{thesis}
M.~Djukanovi{\'c}.
\newblock {\em Split jacobians and lower bounds on heights}.
\newblock PhD thesis, Universiteit Leiden, 2017.

\bibitem{frey}
G.~Frey.
\newblock On elliptic curves with isomorphic torsion structures and
  corresponding curves of genus {$2$}.
\newblock In {\em Elliptic curves, modular forms, \& {F}ermat's last theorem
  ({H}ong {K}ong, 1993)}, Ser. Number Theory, I, pages 79--98. Int. Press,
  Cambridge, MA, 1995.

\bibitem{freykani}
G.~Frey and E.~Kani.
\newblock Curves of genus {$2$} covering elliptic curves and an arithmetical
  application.
\newblock In {\em Arithmetic algebraic geometry ({T}exel, 1989)}, volume~89 of
  {\em Progr. Math.}, pages 153--176. Birkh\"{a}user Boston, Boston, MA, 1991.

\bibitem{16-6}
M.~R. Gonzalez-Dorrego.
\newblock {$(16,6)$} configurations and geometry of {K}ummer surfaces in {${\bf
  P}^3$}.
\newblock {\em Mem. Amer. Math. Soc.}, 107(512):vi+101, 1994.

\bibitem{goursat}
E.~Goursat.
\newblock Sur la r\'{e}duction des int\'{e}grales hyperelliptiques.
\newblock {\em Bull. Soc. Math. France}, 13:143--162, 1885.

\bibitem{harris}
J.~Harris.
\newblock {\em Algebraic geometry}, volume 133 of {\em Graduate Texts in
  Mathematics}.
\newblock Springer-Verlag, New York, 1995.
\newblock A first course, Corrected reprint of the 1992 original.

\bibitem{jacobi}
C.~G.~J. Jacobi.
\newblock Review of {L}egendre's `{T}rait\'{e} des fonctions elliptiques,
  troisi\`{e}me suppl\'{e}ment'.
\newblock {\em J. reine angew. Math.}, 8:413--417, 1832.

\bibitem{kuhn}
R.~M. Kuhn.
\newblock Curves of genus {$2$} with split {J}acobian.
\newblock {\em Trans. Amer. Math. Soc.}, 307(1):41--49, 1988.

\bibitem{kumar}
A.~Kumar.
\newblock Hilbert modular surfaces for square discriminants and elliptic
  subfields of genus 2 function fields.
\newblock {\em Res. Math. Sci.}, 2:Art. 24, 46, 2015.

\bibitem{shaska3}
K.~Magaard, T.~Shaska, and H.~V\"{o}lklein.
\newblock Genus 2 curves that admit a degree 5 map to an elliptic curve.
\newblock {\em Forum Math.}, 21(3):547--566, 2009.

\bibitem{mestre}
J.-F. Mestre.
\newblock Construction de courbes de genre {$2$} \`a partir de leurs modules.
\newblock In {\em Effective methods in algebraic geometry ({C}astiglioncello,
  1990)}, volume~94 of {\em Progr. Math.}, pages 313--334. Birkh\"{a}user
  Boston, Boston, MA, 1991.

\bibitem{milneav}
J.~S. Milne.
\newblock Abelian varieties.
\newblock In {\em Arithmetic geometry ({S}torrs, {C}onn., 1984)}, pages
  103--150. Springer, New York, 1986.

\bibitem{mumford}
D.~Mumford.
\newblock {\em Abelian varieties}, volume~5 of {\em Tata Institute of
  Fundamental Research Studies in Mathematics}.
\newblock Hindustan Book Agency, New Delhi, 2008.
\newblock With appendices by C. P. Ramanujam and Yu. I. Manin. Corrected
  reprint of the second (1974) edition.

\bibitem{shaska2}
T.~Shaska.
\newblock Genus 2 curves with {$(3,3)$}-split {J}acobian and large automorphism
  group.
\newblock In {\em Algorithmic number theory ({S}ydney, 2002)}, volume 2369 of
  {\em Lecture Notes in Comput. Sci.}, pages 205--218. Springer, Berlin, 2002.

\bibitem{shaska}
T.~Shaska.
\newblock Genus two curves covering elliptic curves: a computational approach.
\newblock In {\em Computational aspects of algebraic curves}, volume~13 of {\em
  Lecture Notes Ser. Comput.}, pages 206--231. World Sci. Publ., Hackensack,
  NJ, 2005.

\bibitem{shimura}
G.~Shimura.
\newblock {\em Introduction to the arithmetic theory of automorphic functions},
  volume~11 of {\em Publications of the Mathematical Society of Japan}.
\newblock Princeton University Press, 1971.

\bibitem{aec}
J.~H. Silverman.
\newblock {\em The arithmetic of elliptic curves}, volume 106 of {\em Graduate
  Texts in Mathematics}.
\newblock Springer, Dordrecht, second edition, 2009.

\bibitem{weil}
A.~Weil.
\newblock Zum {B}eweis des {T}orellischen {S}atzes.
\newblock {\em Nachr. Akad. Wiss. G\"{o}ttingen. Math.-Phys. Kl. IIa.},
  1957:33--53, 1957.

\end{thebibliography}
\end{document}